\documentclass[12pt]{article}
\usepackage{algorithm,enumerate}
\usepackage{algpseudocode}
\usepackage[hidelinks]{hyperref}
\usepackage{algorithm,enumerate}
\usepackage{algpseudocode}
\usepackage[hidelinks]{hyperref}

\usepackage{bbm}
\usepackage{amsmath}
\usepackage[hidelinks]{hyperref}
\allowdisplaybreaks
\usepackage{cleveref}
\usepackage[utf8]{inputenc}
\renewcommand{\emph}[1]{\textit{#1}}
\usepackage{enumerate,amsmath,amsthm,latexsym,amssymb}
\usepackage{color}\usepackage{graphicx}

\newcommand{\old}[1]{}

\def\cA{{\mathcal A}}

\def\cC{{\mathcal C}}
\def\cD{{\mathcal D}}
\def\cE{{\mathcal E}}

\def\cS{{\mathcal S}}
\def\cC{{\mathcal C}}

\def\cH{{\mathcal H}}

\newcommand{\set}[1]{\left\{#1\right\}}

\def\cS{\mathcal{S}}
\def\cP{\mathcal{P}}

\def\ii_(#1,#2){i_{#1}^{#2}}

\def\bx{{\bf x} }

\def\bd{{\bf d} }

\def\cM{\mathcal{M}}

\def\d{\delta}

\def\z{\zeta}

\def\l{\lambda}

\def\1{{\bf 1}}
\def\0{{\bf 0}}

\def\pr{\mathbb{P}}

\newcommand\bfrac[2]{\left(\frac{#1}{#2}\right)}

\def\bx{{\bf x} }

\newtheorem{theorem}{Theorem}[section]

\newtheorem{lemma}[theorem]{Lemma}

\newtheorem{claim}[theorem]{Claim}

\usepackage[toc,page]{appendix}

\title{Packing Hamilton Cycles in Cores of Random Graphs}
\author{Michael Anastos}
\begin{document}
\maketitle
\begin{abstract}
Consider the random graph process $\{G_t\}_{t\geq 0}$.
For $k\geq 3$ let $G_{t}^{(k)}$ denote the $k$-core of $G_t$ and let
$\tau_k$ be the minimum $t$ such that the $k$-core of $G_t$ is nonempty.
It is well known that w.h.p.\footnote{We say a sequence of events $\mathcal{E}_n$ holds \emph{with high probability} (w.h.p. for brevity) if $\lim_{n\to \infty} \mathbb{P}(\mathcal{E}_n)=1$.}\@ for $G_{\tau_k}^{(k)}$ has linear size while it is believed to be Hamiltonian. Bollob\'{a}s, Cooper, Fenner and Frieze further conjectured that  w.h.p.\@ $G_{t}^{(k)}$ spans $\lfloor \frac{k-1}{2} \rfloor$ edge-disjoint Hamilton cycles plus, when $k$ is even, a perfect matching for $t\geq \tau_k$. 
We prove that  w.h.p.\@   if $k$ is odd then $G_{t}^{(k)}$ spans $\frac{k-3}{2}$ edge disjoint Hamilton cycles plus an additional 2-factor whereas if $k$ is even  then it spans $\frac{k-2}{2}$ edge disjoint Hamilton cycles plus an additional matching of size $n/2-o(n)$ for $t\geq \tau_k$. In particular  w.h.p.\@ $G_{t}^{(k)}$ is Hamiltonian for $k\geq 4$ and $t\geq \tau_k$. 
This improves upon results of Krivelevich, Lubetzky and Sudakov.
\end{abstract}
\section{Introduction}
Hamilton  cycles is a fundamental object in  graph  theory and it has been studied in both the deterministic and the stochastic setting. One of the first stochastic settings in which the threshold for Hamiltonicity was determined is the  random graph process $\{G_t\}_{t\geq 0}$. $\{G_t\}_{t\geq 0}$ is generated by starting with $G_0$ being the empty graph. Thereafter for $i\geq 1$, given $G_{i-1}$,  $G_i$ is formed by adding to $G_{i-1}$ an edge that is chosen uniformly at random from $\binom{[n]}{2}\setminus E(G_{i-1})$. 

Let 
$$\tau_2=\min\{t:\delta(G_t)=2\}.$$ 
$t\geq \tau_2$ is definitely a necessary condition for $G_t$ to be Hamiltonian. Ajtai, Koml\'{o}s and Szemer\'{e}di \cite{AKS} and Bollob\'{a}s \cite{Bol}, building upon  work of  Korshunov \cite{Kor}, P\'{o}sa \cite{Posa} and Koml\'{o}s and 
Szemer\'{e}di \cite{Komlos},  proved that w.h.p.\@ $G_{\tau_2}$ is Hamiltonian.

W.h.p. $\tau_2=(1+o(1))n \log n/2$. Thus to achieve Hamiltonicity in the random graph process one has to wait until the average degree becomes $(1+o(1))\log n$. 
In order to ``speed up" the appearance of a Hamilton cycle it is natural to consider models of random graphs that ensure that the minimum degree is at least 2. Such a model of random graphs is $G_{n,m}^{\delta\geq k}$, a graph that is chosen uniformly at random from all the graphs on $n$ vertices  with $m\geq kn/2$ edges and minimum degree $k$. Taking $k=2$ does not suffices for  $G_{n,m}^{\delta\geq k}$ to be Hamiltonian when $m=o(n\log n)$. Indeed for $\epsilon>0$ and $m \leq (\frac{1}{6}-\epsilon)n \log n$ w.h.p.\@ $G_{n,m}^{\delta \geq 2}$ contains a 3-spider i.e. a vertex of degree 3 that is incident to 3 vertices of degree 2. Inherently w.h.p.\@ $G_{n,m}^{\delta \geq 2}$ is not Hamiltonian. On the other hand Anastos and Frieze prove that taking $k=3$ and $m\geq 2.67n$ suffices \cite{AF}.    

For a graph $G$ we say that $G\in \cA_k$ if $G$ spans $\lfloor\frac{k-1}{2}\rfloor$ edge disjoint Hamilton cycles plus, when $k$ is even, a perfect matching.
Bollob\'{a}s, Cooper, Fenner and Frieze \cite{BCFF} considered $G_{n,m}^{\d \geq k}$ for $k\geq 3$ and proved the following Theorem.
\begin{theorem}\label{thm:BCFF}
Let $k\geq 3$. There exists a constant $C_k\leq 2(k+1)^3$ such that if $2m\geq C_kn$ then w.h.p.\@ $G_{n,m}^{\d\geq k}\in \cA_k$.  
\end{theorem}
For $c = O(1)$, w.h.p. in $G_{n,cn}^{\delta\geq k}$ there exist vertices whose neighborhood contains $(k+1)$ vertices of degree $k$. Hence the number of edge disjoint Hamilton cycles in the above Theorem is optimal.

A graph that is known to be distributed as $G_{n,m}^{\delta\geq k}$ is the $k$-core of an element of the random graph process. For $k\geq 3$ and a graph $G$ denote by $G^{(k)}$ the $k$-core of $G $ i.e. the maximal subgraph of $G$ of minimum degree $k$.
{\L}uczak showed that the size of $G_t^{(k)}$ goes through a phase transition similar to that of the size of the giant component \cite{l}. Specifically  he showed that  $G_{\tau_k}^{(k)}$ spans either 0 or a constant proportion of the vertices of $G_{\tau_k}$. Later, Pittel, Spencer and Wormald \cite{PSW} established the threshold of $V(G_t^{(k)})\neq \emptyset$ to be at $t=  (c_k/2) n$ where  $c_k=k+\sqrt{k\log k}+o(\sqrt{k})$.
In \cite{BCFF}, Bollob\'{a}s, Cooper, Fenner and Frieze conjecture that Theorem \ref{thm:BCFF}  should extend to every $G_{t}^{ (k)}$ for $t\geq \tau_k$. The first result towards this direction was given by Krivelevich, Lubetzky and Sudakov (see \cite{KLS}). They proved that for $k\geq 15$ w.h.p. $G_{\tau_k}^{(k)}$ is Hamiltonian for  $t\geq \tau_k$. In addition they prove that there exists $k_0$, such that if $k\geq k_0$ then w.h.p.\@ $G_{t}^{(k)}$ spans $\lfloor \frac{k-3}{2} \rfloor$ edge disjoint Hamilton cycles for $k\geq k_0$ and $t\geq \tau_k$.

A 2-factor of a graph $G$ is a 2-regular spanning subgraph of $G$. In this paper we study a very closely related to $\cA_k$ property which we call $\cA_k'$. We say that a graph $G$ has the property $\cA_k'$ if
\begin{itemize}
    \item[(i)] when $k$ is odd, it spans $\frac{k-3}{2}$ pairwise edge-disjoint Hamilton cycles plus a 2-factor,
    \item[(ii)] when $k$ is even, it spans $\lfloor \frac{k-1}{2} \rfloor$ pairwise edge-disjoint Hamilton cycles plus a matching of size $n/2-o(n)$.
\end{itemize}
Hence, $\cA_k'$ is a relaxation of  $\cA_k$ where when $k$ is odd we substitute the last Hamilton cycle with a 2-factor while when $k$ is even we allow for a slightly smaller matching.

The main Theorems of this paper are the following:
\begin{theorem}\label{thm:ham}
Let $4\leq k=O(1)$ and $k/2< c =O(1)$.  Then, $$\pr\big(G_{n,cn}^{\delta\geq k} \in \cA_k'\big)=1-o(n^{-1}).$$
\end{theorem}

\begin{theorem}\label{thm:hamCores}
Let $4\leq k=O(1)$. Then w.h.p. 
$G_{t}^{(k)} \in \cA_k'$  for $t\geq \tau_k$.
\end{theorem} 

In particular
Theorem \ref{thm:hamCores} implies that w.h.p. if the 4-core of  $G_{t}$ is non-empty then it is Hamiltonian. 

To construct the Hamilton cycles in $G\sim G^{\delta \geq k}_{n,m}$ we start by applying  Theorem 6.1 from \cite{A} in order to (i) decompose $G$ into $G'\subset G$ and $R=E(G)\setminus E(G')$ where the distribution of $R$ is fairly close to uniform  and (ii) extract from $G'$ a $(k-1)$-matching $M$ of size $(k-1)n/2-o(n)$. Here by ``a $(k-1)$-matching" we refer to a set of edges that spans a graph of maximum degree $k-1$ as opposed to a matching of size $k-1$.

\begin{theorem}[Theorem 6.1 of \cite{A}]\label{thm:kGreedy2}
Let $k\geq 3$,  $k/2< c=O(1)$, $n^{-0.49}\leq p =o(1)$ and $G\sim G_{n,cn}^{\d \geq k}$. Then, with probability $1-o(n^{-9})$, there exists $V_0\subset V(G)$ of size at most $3cnp$ and $E_p\subset E(G)$ of size at least $\frac{(2cn-kn)p}{4}$ such that
\begin{itemize}
    \item[(i)] Given the set $E(G)\setminus E_p$ the edge set $E_p$ is distributed uniformly at random among all sets of size $|E_p| $ that are disjoint from $E(G)\setminus E_p$ and not incident to $V_0$ and
    \item[(ii)] $E(G)\setminus E_p$ spans a $(k-1)$-matching $M$ of size at least $kn/2-n^{0.401}$.
\end{itemize} 
In addition with probability $1-o(n^{-9})$ the sets $V_0,E_p$ and $M$ described above can be generated in $O(n)$ time.
\end{theorem}

Now given $G',M$ and $R$, promised by the above Theorem, we repeatedly apply the Tutte-Berge formula in order to  peel from $M$, $(k-1)$ matchings of size $n/2-o(n)$, say $M_1,M_2,...,$ $M_{k-1}$. We then, iteratively, convert pairs of matchings into Hamilton cycles as follows. We first take the union of 2 matchings and remove an edge from each cycle created to create a \textbf{VDPC} (vertex disjoint path covering) of $V$. A  VDPC is a set of vertex disjoint path that covers $V$. Here single vertices are considered to be paths of length zero. Thereafter we introduce ``fake edges" and glue the paths given by the VDPC into a Hamilton path. Thereafter, using P\'osa rotations along with few edges from $R$ we close this path into a Hamilton cycle. This new Hamilton cycle either is entirely spanned by $G$ or, by removing a fake edge, it defines a Hamiltonian path with fewer ``fake edges". We repeat this process until we get a Hamilton cycle no ``fake edges" hence a Hamilton cycle that is entirely spanned by $G'\cup R = G$. We slightly abuse the notion of a VDPC and call a Hamilton cycle a VDPC of size 0. 

When applying the above process, after we have constructed Hamilton cycles $H_1,H_2,...,H_i$, to turn $\big(M_{2i}\cup M_{2i+1}\big)\setminus \big(\bigcup_{j=1}^i E(H)\big)$ into a Hamilton cycle $H_{i+1}$ we work in $G'\setminus \big(\cup_{j=1}^i E(H)\big)$. Thus potentially, $E(H_{i+1})$ contains edges from the matchings that have not been processed yet i.e. from $M_j$, $j\geq 2i+2$. We will ensure that at each iteration  $o(n)$ such edges may be used hence at the beginning of each iteration we will have a pair of matching, each of size $n/2-o(n)$. 

The rest of the paper is organized as follows. At Section \ref{sec:structural} we introduce the model which we use to analyse $G_{n,m}^{\delta \geq k}$, prove some typical structural properties of $G_{n,m}^{\delta \geq k}$ and then derive a desired decomposition of it. At Section \ref{sec:gnm} we prove Theorem \ref{thm:ham}. Finally we give a sketch of the proof of Theorem \ref{thm:hamCores} at Section \ref{sec:gnp} which is primarily based on the proof of Theorem \ref{thm:hamCores}.

\section{Structural Properties of $G_{n,m}^{\delta \geq k}$}\label{sec:structural}
\subsection{Generating $G_{n,m}^{\delta \geq k}$}\label{model}
To analyse $G_{n,m}^{\delta \geq k}$ we use a variation of Bollob\'{a}s configuration model \cite{BolCM}. Given $n,m \in \mathbb{N}$ and a sequence of size $2m$, $\bx=(x_1,x_2,...,x_{2m})\in [n]^{2m}$ we define the multigraph $G_\bx$ by $V(G_\bx):=[n]$, $E(G_\bx):=\{\{x_{2j-1},x_{2j}\}:j\in [m]\}$. Thus $G_\bx$ is a graph on $n$ vertices with  $m$ edges. The degree of some vertex $v\in [n]$ with respect to the sequence $\bx$ is equal to the number of times it appears in $\bx$, i.e. $d_\bx(v)=|\{i: x_i=v, 1\leq i\leq 2m\}|$. We let $\cS_{n,2m}^{\delta \geq k}$ be the set of  sequences $\bx=(x_1,x_2,...,x_{2m})$ such that $d_{\bx}(i)\geq k$ for $i \in [n]$.
If $\bx$ is chosen uniformly at random from $\cS_{n,2m}^{\d \geq k}$ then  $G_\bx$ is close in distribution to $G_{n,m}^{\d \geq k}$. Indeed, conditioned on $G_\bx$ being simple, the distributions of $G_\bx$ and $G_{n,m}^{\d \geq k}$ are identical. Both are uniform over the simple graphs on $n$ vertices with $m$ edges and minimum degree $k$. Each such graph will correspond to $m!2^m$ sequences in $\cS_{n,2m}^{\d \geq k}$.

For $\lambda>0$ let 
\begin{equation}\label{eq:fk}
    f_k(\lambda)=e^{\lambda}-\sum_{i=0}^{k-1} \frac{\lambda^i}{i!}.
\end{equation} 
In addition, let $\cP_{\geq k}(\l)$ be the {\em truncated at $k$ Poisson($\lambda$}) random variable, i.e.
$$\pr(\cP_{\geq k}=t)=\frac{{\l}^t}{t!f_k({\l})},\hspace{1in}\text{ for }t\geq k.$$

The next Lemma  describes a typical element of $\cS_{n,2m}^{\d \geq k}$. Let $\bx$ be an element of $\cS_{n,2m}^{\d \geq k}$ chosen  uniformly at random. Lemma \ref{lem:equiv} states that the joint distribution of $d_1,d_2,...,d_{n}$ is the same as the joint distribution of $\cP_1,\cP_2,...,\cP_{n}$ conditioned on $\sum_{i=1}^n \cP_i=2m$, where $\cP_i\sim \cP_{\geq k}(\lambda)$ for $i\in[n]$.

\begin{lemma}\label{lem:equiv}
Let $k,n,m \in \mathbb{N}$ be such that $2m\geq kn$ and let $\bx$ an element of  $\cS_{n,2m}^{\d \geq k}$  chosen uniformly at random. Let $\lambda > 0$ and let $\{Z_i:i \in [n]\}$  be a set of independent $\cP_{\geq k}(\l)$ random variables. Then for every $d_1,d_2,...,d_n\geq k$,
$$ \pr\big(d_\bx(i)=d_i \text{ for } i\in [n]\big)= \pr\bigg(Z_i=d_i \text{ for } i\in [n]\bigg| \sum_{i=1}^n Z_i=2m\bigg).$$
\end{lemma}
\begin{proof}
For $\bx\in \cS_{n,2m}^{\d \geq k}$ let $\cD(\bx)$ be the degree sequence of $\bx$. Define
$\cD=\set{\cD(x):x\in \cS_{n,2m}^{\d \geq k}}$. For a fixed degree sequence $\bd\in \cD$ there exists $(2m)!/\prod_{i\in[n]} \bd_i$ many elements in $\cS_{n,2m}^{\d \geq k}$ with that degree sequence. Thus, for $\bd \in \cD$,
\begin{align*} 
\Pr(\cD(\bx)=\bd)&=  \bfrac{(2m)!}{\underset{i\in [n]}{\prod}\bd_i!} 
\bigg/
\left( \sum_{\bx \in \cS_{n,2m}^{\d \geq k} } 1\right)
=\left( \frac{(2m)!}{\underset{i\in [n]}{\prod}\bd_i!}\right)
\bigg/
\left( \sum_{\bd' \in \cD }\frac{(2m)!}{\underset{v\in [n]}{\prod}\bd_i'!} \right).
\end{align*}
On the other hand, 
\begin{align*}
\Pr \bigg( (Z_1,Z_2,...,Z_n)=\bd \bigg|\sum_{i=1}^n  Z_i = 2m \bigg)
&=  \left( \underset{i\in [n]}{\prod}\frac{ 
e^{-\lambda}\lambda^{\bd_i} }{\bd_i! f_{k}(\lambda)}  \right) \bigg/ \left( \sum_{\bd' \in \cD} \underset{i\in [n]}{\prod}
\frac{ e^{-\lambda}\lambda^{\bd_i'}}{\bd_i'! f_{k}(\lambda)  }\right)
\\& = \left( \lambda^{2m} \underset{i\in [n]}{\prod}
\frac{1}{\bd_i!}  \right)  \bigg/ \left( \lambda^{2m} \sum_{\bd' \in \cD} \underset{i\in [n]}{\prod}
\frac{1}{\bd_i'!} \right) 
\\&=\Pr(\cD(\bx)=\bd).
\end{align*}
\end{proof}

It can be shown, see for example \cite{McK} that for a random $x\in \cS_{n,2m}^{\d\geq k}$ if
$m=O(n)$ then,
\begin{equation*}
\Pr(G_\bx\text{ is simple})=\Omega(1).
\end{equation*}
Hence, choosing a random element of $\bx\in \cS_{n,2m}^{\d\geq k}$ and then generating $G_\bx$ is a good model for generating
$G_{n,m}^{\d \geq k}$ and for any function $f(\cdot)$ such that $f(n)\to 0$ as $n\to \infty$ any properties that hold with probability $1-o(f(n))$ for $G_\bx$ also hold with probability $1-o(f(n))$ for $G_{n,m}^{\d \geq k}$.

\subsection{Expansion Properties of $G_{n,m}^{\d\geq k}$}

Let $k/2<c=O(1)$ and $m=cn =O(n)$. Let $\lambda$ be the unique positive real number that satisfies 
\begin{equation}\label{eq:lambda}
 \frac{{\l}f_{k-1}({\l})}{f_k({\l})}=2m.
\end{equation}

Let $\cE$ be an occupancy event in $G_{n,m}^{\d\geq k}$. Denote by $G_{n,m}^{\d\geq k,seq}$ the random graph that is generated from the random sequence model (i.e. from choosing a random element of $\cS_{n,2m}^{\d\geq k}$ and then generating the corresponding graph) and $G_{n,m}^{\d\geq k,Po(\l)}$ the random graph that is generated by first generating $n$ independent, $ \cP_{k}(\l)$ random variables $P_1,P_2,...,P_n$, then choosing a random sequence in $[n]^{\sum_{i\in [n]} P_i}$ with degree sequence $P_1,P_2,...,P_n$ and finally generating he corresponding graph if $\sum_{i\in [n]} P_i$ is even. Then,
\begin{align}\label{eq:models}
    \Pr\big(G_{n,m}^{\d\geq k} \in \cE\big) &\leq O(1) \Pr\big(G_{n,m}^{\d\geq k,seq} \in \cE\big)= O(1) \Pr\bigg(G_{n,m}^{\d\geq k,Po(\l)} \in \cE\bigg| \sum_{i\in [n]}{P_i}=2m\bigg) \nonumber
 \\&\leq O(n^{0.5}) \Pr\big(G_{n,m}^{\d\geq k,Po(\l)} \in \cE\big),
\end{align} 
where the last inequality  in \eqref{eq:models} follows by the choice of $\lambda$.

We summarize the expansion properties of $ G_{n,m}^{\delta \geq k}$ at the next lemma. Its proof is given in Appendix \ref{app:lemmas}.

\begin{lemma}\label{lem:expansion}
Let $G\sim G_{n,m}^{\delta \geq k}$, where $m=cn$ with $k/2<c=O(1)$ and let $\lambda$ be given by \eqref{eq:lambda}. Let $\beta_1,\gamma_1\in (0,0.1)$ be such that $\bfrac{9e^{1+\lambda} \lambda^2}{cf_k(\lambda)} \bfrac{\gamma_1 \lambda}{c}^{0.1}<\frac{1}{2}$ and $[2(k+\lambda) + \log_2(\beta_1\gamma_1)+ 3]\beta_1  < 2(1-\beta_1)$. Then with probability $1-o(n^{-1})$,
\begin{itemize}
\item[(i)] every set $S\subset V(G)$ of size $ |S| \leq \gamma_1n$  spans less than $1.1|S|+1$ edges,
\item[(ii)] every set $S\subset V(S)$ of size  $|S| \leq \beta_1 \gamma_1 n$  is incident to less than $2(1-\beta_1)\gamma_1 n$ edges,
\item[(iii)] $G$ does not span a set of $\frac{2n}{(\log\log n)^6}$ vertex disjoint cycles.
\end{itemize}
\end{lemma}

\subsection{Decomposing $G_{n,m}^{\d\geq k}$}
To pill off matchings of size $(1+o(1))n/2$ from the large $k$-matching promised by Theorem \ref{thm:kGreedy2} we use the following Lemma.
\begin{lemma}\label{lem:pillmatch}	
Let $r\geq 0$ and $\ell\in\mathbb{N}^+$ . Let $G=(V=[n],E)$ be a graph of maximum degree $\ell$ with $|E|\geq \ell n/2- \frac{rn}{(\log\log n)^6}$.
If $G$  does not span a set of $\frac{2n}{(\log\log n)^6}$ vertex disjoint cycles then it spans a matching of size at least $n/2-\frac{(r+2)n}{2(\log\log n)^6}$.
\end{lemma}
\begin{proof}
The Tutte-Berge formula states that the maximum matching of $G$, denoted by $\alpha'(G)$, is given by
\begin{equation}\label{eq:T-B}
2\alpha'(G)=\min_{S\subset V(G)}\{n+|S|-o(G-S)\},
\end{equation}
where by  $o(G-S)$ we denote the number of odd components in $G-S$.
Let $S^*$ be a set of maximum size for which $n+|S|-o(G-S)$ is minimized. Observe that every odd component in $G-S^*$ that is a tree has size 1 i.e. it is an isolated vertex of $G-S$. Indeed, if an odd component $C_i$ of $G-S$ is a tree of size larger than 1 (hence of size at least 3) then by letting $L$ to be the set of leaves of $C_i$, $R$ their neighbors in $C_i$ and $S'=S^*\cup R$ we have the following:
If $|L|=|R|+i$ for $i=0,1$ then,  
$$n+|S'|-o(G-S')= n+|S^*|+|R|-o(G-S^*)-|L|+i= n+|S^*|-o(G-S^*),$$ contradicting the maximality of $S^*$. Otherwise, $$n+|S'|-o(G-S')< n+|S^*|+|R|-o(G-S^*)-|R|,$$ contradicting that $S^*$ has been chosen to minimize $n+|S|-o(G-S)$.

Now let $n_{s}$  be the number of isolated vertices in $G-S$ and $n_l$ be the number of the  odd components in $G-S$ that span a cycle. Then, as no odd component of $G\setminus S$ is a tree on more than 2 vertices, we have that $o(G\setminus S)=n_s+n_l$ and $n_l$ is bounded by the maximum number of vertex disjoint cycles in $G$ which is by assumption at most $\frac{2n}{(\log\log n)^6}$.
Thereafter, as $G$ has maximum degree $\ell$ and spans at least  $\ell n/2- \frac{rn}{(\log\log n)^6}$ edges, by considering the edges between $S$ and the isolated vertices in $G-S$ we have,
$$ \ell n_s - \frac{rn}{(\log\log n)^6} \leq \ell |S|.$$
Thus, 
$$2\alpha'(G)= n +|S| - n_s-n_l\geq n- \frac{rn}{\ell(\log\log n)^6}-\frac{2n}{(\log\log n)^6}  \geq n-\frac{(r+2)n}{(\log\log n)^6}.$$
\end{proof}

\begin{theorem}\label{lem:kGreedy2}
Let $k\geq 3$,  $k< c=O(1)$ and $G\sim G_{n,cn}^{\d \geq k}$. Then, with probability $1-o(n^{-2})$, there exist $E_R\subset E(G)$ of size $\frac{(k-1)n}{\log\log n}$ and $E'\subset \binom{[n]}{2}$ of size $o(n^2)$ satisfy the following: 
\begin{itemize}
    \item[(i)] $E_R$ is distributed uniformly over all sets of size $\frac{(k-1)n}{\log\log n}$ that are subset of\\ $\binom{[n]}{2}\setminus \big(E'\cup E(G)\setminus E_R \big)$
\item[(ii)] $E(G)\setminus E_R$ spans a set of $k-1$ edge disjoint matchings $M_1,M_2,...,M_{k-1}$ each having size at least $\frac{n}{2}-\frac{2^kn}{(\log\log n)^6}$.
\end{itemize} 
\end{theorem}
\begin{proof}
We first apply Theorem \ref{thm:kGreedy2} with $p=\frac{n}{\log\log \log n}$. We let $E_R$ be a random subset of $E_p$ of size $\frac{(k-1)n}{\log\log n}$ and $E'$ be the set of edges incident to the set $V_0$ (of size $O(np)$) given by Theorem \ref{thm:kGreedy2}. Then Theorem \ref{thm:kGreedy2} implies that with probability $1-o(n^{-2})$, $G$, $E_R$, $E'$ satisfy Condition (i) while  $G\setminus (E'\cup R)$ spans a $k$-matching $M$ of size $kn/2-o(n^{0.41})$. Let $H$ be the graph spanned by $M$.

We let $H_1=H$. For $i\in [k-1]$, having defined a subgraph $H_i$ of $H$ of maximum degree $k-i+1$ and of size at least $\frac{(k-i+1)n}{2}-\frac{2^{i}n}{(\log\log n)^6}$ we let $M_i\subset E(H_i)$ be a maximum matching of $H_i$.
Lemma \ref{lem:expansion} implies that $H_i$ satisfies the conditions of Lemma \ref{lem:pillmatch} with $r=2^i$ and $\ell=k-i+1$. Thus, Lemma \ref{lem:pillmatch} implies that $|M_i|\geq n/2-\frac{2^{i}n}{(\log\log n)^6}$.

Now to construct $H_{i+1}$, we first remove form $H_i$ the matching $M_i$,  and then we remove an edge incident to every vertex of $H_i$  that is not saturated by $M_i$. Hence $H_{i+1}$ has maximum degree $k-i$ and spans at least $E(H_i)-|M_i|-(n-2|M_i|) \geq \frac{(k-i)n}{2}-\frac{2^{i+1}n}{(\log\log n)^6}$ many edges. 
\end{proof}

\section{Packing Hamilton Cycles in $G_{n,m}^{\d\geq k}$}\label{sec:gnm} 
In this section we prove Theorem \ref{thm:ham}. The main tool that we are going to use are P\'osa rotations. P\'osa rotations is a procedure that starts with a path and aims to either find a path of larger length or many paths of the same  length. Given a path $P=(x_1,x_2,\ldots,x_s)$ and an edge $\set{x_s,x_i}$ with $1<i<s-1$, the path $(x_1,\ldots,x_i,x_s,x_{s-1},\ldots,x_{i+1})$ is said to be obtained from $P$ by a P\'{o}sa rotation that fixes the end-vertex $x_1$. In such a case we call the vertex $x_i$ the pivot vertex, $x_ix_{i+1}$ the deleted edge and $x_ix_s$ the inserted edge.

We will repeatedly apply the following lemma to subgraphs of $G^{\delta \geq k}_{n,m}$ in order to construct the Hamilton cycles one by one.

\begin{lemma}\label{lem:builtingmatch}
Let $G=(V=[n],E)$ be a graph,  $E',E_R\subset \binom{[n]}{2}\setminus E$,  and $\beta,\gamma,\epsilon \in (0,1)$ be such that 
\begin{itemize}
    \item[(i)] $G$ has minimum degree $4$,
    \item[(ii)] $E$ spans a 2-matching $M$ of size $n-\frac{n}{(\log\log n)^{6-\epsilon}}$, 
    \item[(iii)] $G$ does not spans a set of $\frac{n}{(\log\log n)^6}$ pairwise vertex disjoint cycles,
\item[(iv)] $|E'|=o(n^2)$ and
 $|E_R| = \frac{n}{\log\log n }$,
\item[(v)]  $E_R$ is distributed uniformly over the subsets of $\binom{[n]}{2}\setminus (E\cup E')$ of size $|E_R|$,
\item[(vi)] every set $S\subset V$ of size less than $\gamma n$ spans less than $1.1|S|+1$ edges,
\item[(vii)] every set $S\subset V$ of size less than $ \beta \gamma n$ is incident to less than $2(1-\beta)\gamma n$ edges.
\end{itemize}
Then, with probability $1-o(n^{-1})$, $G \cup E_R$ spans a Hamilton cycle $H$ that intersects $M$ in at least $n- \frac{n}{(\log \log n)^{6-\epsilon-1/10k }}$ edges. 
\end{lemma}

\begin{proof}
Let $M$ be a maximum $2$-matching of $G$,
$\tau=\frac{n}{\log\log n}$ and $E_R=\{e_1,e_2,...,e_{\tau}\}$. 
Properties (ii) and (iii) imply that upon removing an edge from every cycle, $M$ defines a VDPC, say $\cP_0$, of size at most $\frac{2n}{(\log\log n)^{6-\epsilon}}$ that intersects $M$ in at least $|M|-\frac{2n}{(\log\log n)^{6-\epsilon}}$ edges.

For $0\leq t \leq \tau$ let  $G_t=G\cup \{e_1,e_2,...,e_{t}\}$ and $s_t$ be the minimum size of a VDPC of $G_t$ that intersects $M$ in at least $|M|-r_t$ edges, where $r_t$ is defined as follows. $r_0 = \frac{2n}{(\log\log n)^{6-\epsilon}}$ and therefore $s_0\geq 2r_0$. We also let $s_{-1}=n$. Thereafter, for $t\geq 1$ if $s_{t-1}=s_{t-2}$ then $r_{t}=r_{t-1}$. Else,
\begin{align*}
    r_t=r_{t-1}+ 
    \begin{cases}
(\log\log\log n)^2 &\text{ if  } s_t\geq \frac{n}{(\log\log n)^8},
\\(\log\log n)^2 &\text{ if  }  \frac{n}{(\log n)^8} <s_t < \frac{n}{(\log\log n)^8},
\\(\log n)^2 &\text{ if  } s_t\leq \frac{n}{(\log n)^8}.
\end{cases}
\end{align*}
As $s_t$ is decreasing
$r_t \leq r_0+ s_0(\log\log\log n)^2$ for $t\geq 0$ and if $s_{t}=0$ for some $t>0$ then $G_{t}$ spans a Hamilton cycle that intersects $M$ in at least $|M|-r_t\geq   n-\frac{n}{(\log\log n)^{6-\epsilon-1/10k}}$ edges for sufficiently large $n$. 

Now let $t\geq 1$, $\cP_t=\{P_1,P_2,\ldots,P_{s_t}\}$ be a VDPC of $G_t$ of size $s_t$ that intersects $M$ in at least $|M|-r_t$ edges,  $l_t=\log_{1.1}\frac{n}{\max\{s_t-1,1\}}+4$ and $l_t'=l_t +\log_{1.1}\log n$. For $P_i \in \cP_t$ let $v_{i,1},v_{i,2}$ be its two endpoints. For each pair $i\neq j \in [s_t]$ we introduce a set of edges $F_{i,j}$ of size $s_t-1$ such that  $F_{i,j} \cup \big(\bigcup_{P\in \cP_t} E(P)\big)$ spans a Hamilton path $H_{i,j}$ from $v_{i,1}$ to $v_{j,2}$.
We let $V_{left}=\{v_{i,1}:i \in [s_t]\}$ and for $v=v_{i,1}\in V_{left}$ we denote by $\cH_{t,v,left}$ the set of Hamilton paths $\{P_{i,j}:j\in[s_t]\setminus\{i\}\}$ if $s_t>1$. Otherwise we let $\cH_{t,v,left}=\{P_1\}.$  

Thereafter, for $v\in V_{left}$ we let $\cH_{t,v,left,l_t}$ be the set of Hamilton paths that can be obtain from some path in $\cH_{t,v,left}$, via at most $l_t$ P\'{o}sa rotations that fix the vertex $v$ and with the restriction that the inserted edges do not belong to $F$. Having generated the sets $\cH_{t,v,left,l_t}$ we let $V_{right}\subseteq V$ be the set of vertices $v'\in V$ for which there exists at least $\max\{\frac{s_t}{\log n},1\}$ sets $\cH_{t,v,left,l_t}$, $v\in V_{left}$ containing a Hamilton path from $v$ to $v'$. We then let for $v'\in V_{right}$, $\cH_{t,v,right}$ be a set containing $\max\{\frac{s_t}{\log n},1\}$ of those Hamilton paths with pairwise distinct endpoints. 

We then let $\cH_{t,v',right,l_t'}$ be the set of Hamilton paths that can be obtained from some path in $\cH_{t,v',right}$ via at most $l_t'$ P\'{o}sa rotations that fix the vertex $v'$ and with the restriction that the inserted edges do not belong to $F$. Finally we let $\cH_t$ be a maximal set of Hamilton paths in $\cup_{v'\in V_{right}}\cH_{t,v',right,l_t'}$ with pairwise distinct endpoints. The claim that shortly follows will be utilized to lower bound the size of $\cH_t$.  

For $\ell \geq 0$, $dir\in \{left,right\}$ and $v\in V_{dir}$ we let $End_{v,\ell,dir}$ be the set of endpoints of Hamilton paths in $G\cup F$ that can be obtain from some path in $\cH_{t,v,dir}$ via at most $\ell$ P\'{o}sa rotations that fix the vertex $v$ and with the restriction that the inserted edges do not belong to $F$. In addition we let $Pivot_{v,\ell,dir}$ be the  corresponding set of $Pivot$ vertices. 
\begin{claim}\label{claim:Ham:2}
For $\ell \geq 0$, $dir\in \{left,right\}$ and $v\in V_{dir}$ at least one of the following holds,
\begin{itemize}
\item[(i)] $\ell \leq 4$,
\item[(ii)] $|End_{v,\ell+1,dir}|\geq 1.1|End_{v,\ell,dir}|$,
\item[(iii)] $End_{v,\ell+1,dir}\cup  Pivot_{v,\ell+1,dir}$ spans at least  $1.1|End_{v,\ell+1,dir}\cup  Pivot_{\ell+1}|$ edges in $G$,
\item[(iv)] some endpoint in $|End_{v,\ell+1,dir}|$ is obtained via a P\'osa rotation at which an edge in $F$ is deleted.
\end{itemize}
\end{claim}
\textbf{Proof of Claim \ref{claim:Ham:2}:}
Fix $\ell \geq 0$, $dir\in \{left,right\}$ and $v\in V_{dir}$ and assume that (iv) does not hold. Let  $R_1=End_{v,\ell,dir}$, $R_2=End_{v,\ell+1,dir} \setminus End_{v,\ell,dir}$, $R=R_1\cup R_{2}=End_{v,\ell+1,dir}$ and $L=Pivot_{v,\ell+1,dir} \setminus R$. Let $S$ be the subgraph of $G_t$ induced by $R \cup L$. 

Let $u\in End_{v,\ell,dir}$ and $P_{v,u}$ be a $v-u$ Hamilton path. At a P\'osa rotation applied to $P_{v,u}$, that may follow, one of the at least 3 edges incident to $u$ in $E(G)\setminus E(P_{v,u})$ may be chosen to be inserted making its other endpoint, say $r$, a pivot vertex. Thereafter an edge incident to $r$ in $E(P_{v,u})$ will be removed resulting to a possibly new endpoint. As Condition (iv) of Claim \ref{claim:Ham:2} does not apply, the removed edge belongs to $E(G)$.
Thus, in $S$ every vertex in $R_{1}$ and $R_{2}$ respectively has degree at least 3 and 1 respectively while every vertex in $L$ is adjacent to at least 2 vertices in $R$.

Now if $|End_{v,\ell+1,dir}|\leq 1.1|End_{v,\ell,dir}|$ and $|R| > 0.6|L|$ then,  
$$ \frac{2|E(S)|}{|V(S)|} \geq \frac{2|L|+3\cdot 0.9|R|+0.1|R|}{|L|+|R|} > \frac{2+3\cdot 0.9\cdot 0.6+0.1\cdot 0.6}{1+0.6}= 2.3.$$

On the other hand if $|End_{v,\ell+1,dir}|\leq 1.1|End_{v,\ell,dir}|$ and $|R| \leq 0.6 |L|$ then,
$$ \frac{|E(S)|}{|V(S)|} \geq \frac{2|L|}{1.6|L|}>1.25.$$

Thus if neither of Conditions (ii), (iv) of  Claim \ref{claim:Ham:2} apply  then $R\cup L$ either spans at least $1.1|R \cup L|+1$ many edges (i.e. Condition (iii) holds) or $|R\cup L| \leq 9$. Now recall that $|Pivot_{v,1,dir}|=|End_{v,1,dir}|=d(v)-1\geq 2$ and in the graph spanned by $S'=End_{v,5,dir}\cup Pivot_{v,5,dir}$ every vertex in $End_{v,4,dir}$ has degree at least 3 while every vertex in $Pivot_{v,5,dir}$ has degree at least 2. Thus $|V(S')|\geq 10$ and if $|R\cup L| \leq 9$ then $\ell \leq 4$. 
\qed

First assume that in the process of generating $\cH_t$ at some P\'osa rotations, an edge from $F$ was removed resulting to a Hamilton path $P$. Then, as $P$ was generated via a sequence of at most $(l_t+l_t')$ P\'osa rotations, $E(P)\setminus F$ defines a path covering of size $s_t-1$ that intersects $M$ in at least $|M|-r_t-2(l_t+l_t')$ edges. 

Otherwise, Claim \ref{claim:Ham:2} together with Condition (vi) implies that either $|End_{l_t,v,left}|\geq 1.1^{l_t-4}(s_t-1) \geq n$ or $|End_{l,v,left} \cup Pivot_{l,v,left}| \geq \gamma n$ for some $l\leq l_t$. In the second case, as every vertex in $Pivot_{l,v,left} \setminus End_{l,v,left}$ has at least 2 neighbors in $End_{l,v,left}$ Condition (vii) implies that $|End_{l_t,v,left}| \geq |End_{l,v,left}| \geq \beta\gamma n$.

Hence,
$$|V_{right}| \geq \frac{s_t\cdot \beta \gamma n-\frac{s_t}{\log n}\cdot n}{s_t} \geq \frac{\beta \gamma n}{2}.$$
Thereafter, Claim \ref{claim:Ham:2} implies that $|End_{l_t',v,right}|\geq \beta\gamma n$ for $v\in V_{right}$ and therefore,
$$\cH_t \geq \frac{(\beta\gamma)^2n}{4}.$$ 

For $P \in \cH_t$ let $P_1,P_2$ be its endpoints and $Q_t=\set{\{P_1,P_1\}:P\in \cH_t}$. As $P\in \cH_t$ was generated via a sequence of at most $(l_t+l_t')$ P\'osa rotations, $(E(P) \cup \{P_1,P_2\}) \setminus F$ defines a path covering of size $s_t-1$ that intersects $M$ in at least $|M|-r_t-2(l_t+l_t')$ edges (since after every P\'osa the ``current" path differs in 2 edges from its ``predecessor"). 

. 

Thus,
\begin{align*}
2(l_t+l_t')\leq
\begin{cases}
(\log \log \log n)^2 &\text{ if  } s_t\geq \frac{n}{(\log \log n)^8},
\\(\log \log n)^2 &\text{ if  } \frac{n}{(\log n)^8}< s_t < \frac{n}{(\log \log n)^8},
\\(\log n)^2 &\text{ if  }  s_t \leq \frac{n}{(\log n)^8}.
\end{cases}
\end{align*}
Hence,
\begin{align}\label{bound_prob_rot1}
    \Pr(s_{t+i}< s_t)\geq \Pr(e_{t+i}\in Q_t \setminus E')\geq (1+o(1))(\beta \gamma)^2/2,
\end{align}
and
\begin{align*}
    \Pr(s_\tau >0) &\leq \Pr(Binomial(\tau,(\beta \gamma)^2/3)\leq \tau/(\log\log n)^3)+o(n^{-1})
\\& \leq  \binom{\tau}{\tau/(\log\log n)^3}\bigg(1-\frac{(\beta\gamma)^2}{3} \bigg)^{(1+o(1))\tau}+o(n^{-1})
\\&\leq  \bigg(e(\log\log n)^3)\bigg)^{\frac{\tau}{(\log\log n)^3}}\cdot e^{-0.3(\beta\gamma)^2\tau} o(n^{-1})=o(n^{-1}).
\end{align*}
Hence with probability $1-o(n^{-1})$ we have that $s_{\tau}=0$ yielding a VDPC of size 0, hence a Hamilton cycle in $G\cup E_R$, that intersects $M$ in at least $|M|-\frac{n}{(\log\log n)^{6-\epsilon-10/k}}$ edges.
\end{proof}

We will use the following Lemma to augment the final 2-matching to a 2-factor.

\begin{lemma}\label{lem:builting2factor}
Let $G=(V=[n],E)$ be a graph,  $E',E_R\subset \binom{[n]}{2}\setminus E$,  and $\beta,\gamma,\epsilon \in (0,1)$ be such that 
\begin{itemize}
    \item[(i)] $G$ has minimum degree $3$,
    \item[(ii)] $E$ spans a 2-matching $M$ of size $n-\frac{n}{(\log\log n)^{4}}$,
\item[(iii)] $|E'|=o(n^2)$ and
 $|E_R| = \frac{n}{\log\log n }$,
\item[(v)]  $E_R$ is distributed uniformly over the subsets of $\binom{[n]}{2}\setminus (E\cup E')$ of size $|E_R|$,
\item[(vi)] every set $S\subset V$ of size less than $\gamma n$ spans less than $1.1|S|+1$ edges.
\end{itemize}
Then, with probability $1-o(n^{-1})$, $G \cup E_R$ spans a 2-factor. 
\end{lemma}

\begin{proof}
Given a $2$-matching $M'$ of $G$, we say that the path $P=v_0,e_1,v_1,.....,e_s,v_s$ is  $M'$-alternating if its odd indexed edges do not belong to $M'$ whereas its even indexed edges do (here we slightly abuse the traditional definition of alternating paths where $E(P)\cap M$ consists either of the odd or of the even indexed edges of $P$). We say that $P$ is $M'$-augmenting if it is an $M'$-alternating path  of odd length. Hence if $P$ is $M'$-augmenting then  $M'\triangle E(P)$ is a $2$-matching of size $|M'|+1$. In addition for a $2$-matching $M'$ and $v\in V(G)$ we denote by $d_{M'}(v)$ the number of edges that are incident to $v$ in $M'$.

Let $\tau=\frac{n}{\log\log n}$ and  $E_R=\{e_1,e_2,...,e_{\tau}\}$. For $0\leq t \leq \tau$ let  $G_t=G\cup\{e_1,e_2,...,e_{t}\}$ and $M_t$ be a maximum $2$-matching of $G_t$. If $|M_t|<n$ let $v,w\in [n]$ be such that $d_{M_t}(v), d_{M_t}(w)\leq 1$. In the case that $d_{M_t}(v)=0$ we may let $w=v$. For $u\in V(G)$ we let $P_u$ be the shortest $M_t$-alternating path from $v$ to $u$ if such a path exists, otherwise we let $P_u=\emptyset$. Define the sets $$Q_v:=\{u\in V(G): v=u \text{ or }|P_u|= 0\text{ mod }2 \text{ and }P_u \neq \emptyset\}$$ and $$W_v:=\{u\in V(G): |P_u|= 1\text{ mod } 2\}.$$ Let $S$ be the subgraph of $G_t$ induced by $Q_v\cup W_v$. As $M_t$ is maximum every vertex $u\in W_v$ is incident to 2 edges in $M_t$ (otherwise $P_u$ is $M_t$-augmenting) and has 3 neighbors in $V(S)$, 1 defined by $P_u$ and 2 defined by $M_t$. Moreover every vertex in $Q_v$ has at least 2 neighbors in $V(S)$ and at least 1 in $W_v$, one defined by $M_t$ and all of its neighbors via edges not in $M$. Thus either $|Q_v| \geq 0.3|W_v|$ and 
$$\frac{2|E(S)|}{|V(S)|} \geq \frac{(2+0.3\cdot 3)|W_v|}{(1+0.3)|W_v|}> 2.23$$ or 
$|Q_v| \leq 0.3|W_v|$ and 
$$\frac{|E(S)|}{|V(S)|} \geq \frac{|W_v|+|W_v|/2}{1.3|W_v|}\geq 1.15$$
Thus either  $V(S)$ spans at least $1.1|S|+1$ edges and therefore $|S| \geq \gamma n$  or $|S| \leq 34$. Now observe that as every vertex in  $Q_v$ is incident to 2 edges in $M_t$ the number of vertices $u$ for which $P_u$ has length $\ell$ is at least $2,4,2,4,4,8,8$ and $16$ for $\ell=1,2,3,4,5,6$ and $7$ respectively. Thus $|S|>34$ and therefore  $|S|\geq \gamma n.$

Now let $Q_v'=Q(v,M_t,G_t)$ be the set of vertices that are reachable from $v$ via an $M_t$-alternating path of even length. Observe that if $z\in N(Q_v')$ then $z$ is incident to  some vertex in $Q_v'$ via an edge in $M_t$ and hence $|N(Q_v')| \leq 2|Q_v'|$. Indeed, assume otherwise. Then there exist  $z\in N(Q_v')$ and $u\in Q_v'$ such that $\{u,z\} \in G_t\setminus M_t$ and $z$ does not have an $M_t$-neighbor in $Q_v'$. The edge $\{u,z\}$ gives rise to an $M_t$-alternating path $P$ from $v$ to $u$ to $z$. Now if $d_{M_t}(z)=0$ then $P$ is $M_t$-augmenting contradicting the maximality of $M_t$. Otherwise there exists some edge $\{z,z'\}\in M_t$. In such a case the path $P,\{z,z'\},z'$ witnesses the candidacy of $z'$ in $Q_v'$ which gives a contradiction. 

Finally observe that $V(S)\subset Q_v'\cup N(Q_v').$ As $|N(Q_v')| \leq 2|Q_v'|$ we have that $|Q_v'| \geq |S|/3 \geq \gamma n/3.$

For every vertex $u\in Q_v'$ the underlying $M_t$-alternating path $P_{v,u}$ from $v$ to $u$ defines a maximum $2$-matching $M_u=M_t\triangle E(P_{v,u})$ of $G_t$ such that $d_{M_u}(u),d_{M_u}(w)\leq 1$. Now, by repeating the same argument with $M_u$ in place of $M_t$ and $w$ in place of $v$ we can define in a similar manner the set $Q_{u,w}'$ (in place of $Q_v'$). This gives a set $\mathcal{M}$ of at least $\gamma^2n^2/18$ couples $(\{x,y\},M_{\{x,y\}})$ where $x\in Q_v'$, $y\in Q_{x,w} \cap V_1$, $M_{\{x,y\}}$ is a maximum $2$-matching of $G_t$ and $d_{M_{\{x,y\}}}(x),d_{M_{\{x,y\}}}(y)\leq 1$. Thus if $e_{t+1}=e$ for some $\{e,M_e\} \in \cM$ then $\{e\}\cup M_e$ is a 2-matching of  $G_{t+1}$ of size $|M_t|+1>|M_t|$.

Hence, the probability that $G_\tau$ does not span a 2-factor is bounded above by
\begin{align*}
 \Pr\bigg( Bin\bigg(\frac{n}{\log\log n},\frac{\gamma^2 n^{2}}{18}\bigg) \leq \frac{n}{(\log\log n)^4} \bigg)=o(n^{-1}).
\end{align*}

\end{proof}

\textbf{Proof of Theorem \ref{thm:ham}}:
 We let $M_1$, $M_2$,...,$M_{k-1}$, $E_R$,$E'$,  $G'=G\setminus E_R$ be the matchings, edge sets and graph promised by Lemma \ref{lem:kGreedy2}. We  randomly partition $E_R$ into $k-1$ sets $R_1,R_2,...,R_{k-1}$ of size $\frac{n}{\log\log n}$. Having constructed  Hamilton cycles $H_1,H_2,...,H_{i-1}$, $i< (k-3)/2$ such that $E(H_j)\setminus (M_{2j-1}\cup M_{2j})$ has size at most $\frac{n}{(\log\log n)^{6-j/k}}$ for $j\in[i-1]$ we construct a Hamilton cycle $H_i$ with $E(H_i)\subset E(G)\setminus \big(\cup_{j\in[i-1]}E( H_j)\big)$ such that $E(H_i)\setminus (M_{2i-1}\cup M_i)$ has size at most $\frac{n}{(\log\log n)^{6-i/k}}$.
 
 For that we apply Lemma \ref{lem:builtingmatch} with 2-matching $M_i'=(M_{2i-1}\cup M_{2i})\setminus \big(\cup_{j\in[i-1]} E(H_j)\big)$, graph of minimum degree 4 $G_i'=G'\setminus \big(\cup_{j\in[i-1]} H_j\big)$, the set of random edges $R_i$, the set of forbidden edges $E_i'=E'\cup   
 \big(\cup_{j\in[i-1]} E(H_j)\big)
 \cup \big(\cup_{j\in[i-1]} R_j\big)$, $\beta=\beta_1$, $\gamma=\gamma_1$  (as in Lemma \ref{lem:expansion}) and $\epsilon= 9i/10k$. Then, $|M_i'|\geq |M_{2i-1}|+|M_{2i}|- \sum_{j=1}^{i-1} \frac{n}{(\log\log n)^{6-j/k}} \geq n- \frac{n}{(\log\log n)^{6-\epsilon}}.$ Lemma \ref{lem:expansion} implies that the rest of the conditions of Lemma \ref{lem:builtingmatch} hold with probability $1-o(n^{-1})$.
 Thus with probability $1-o(n^{-1})$, $G_i'\cup R_i$
 spans a Hamilton cycle $H_i$ that intersects $M_i'$ in at least $n-\frac{n}{(\log\log n)^{6-9i/10k+i/10k}}= n-\frac{n}{(\log\log n)^{6-i/k}}.$
 
 Finally if $k$ is even then $M_{k-1}\setminus \big(\cup_{j\in[(k-2)/2]} E(H_j)\big)$ is a matching of size  at least
 $$n/2-O\bigg(\sum_{j=1}^{(k-2)/2} \frac{n}{(\log\log n)^{6-j/k}}\bigg)=n/2-o(n).$$
On the other hand, if $k$ is odd then  
$M_{last}=M_{k-2} \cup M_{k-1}\setminus \big(\cup_{j\in[(k-3)/2]} H_j\big)$ is a 2-matching of size $n-o\bfrac{n}{(\log \log n)^4}$. Let $G_{last}=G'\setminus \big(\cup_{j\in[(k-3)/2]}E( H_j)\big)$. Then $G_{last}$ has minimum degree 3 and 
Lemma \ref{lem:builting2factor} implies that $G_{last}\cup R_k$ spans a 2-factor with probability $1-o(n^{-1})$. 
\qed

\section{Packing Hamilton Cycles in  $G_{t}^{(k)}$}\label{sec:gnp}

Recall we denote by $G_0,G_1,...,G_{\binom{n}{2}}$ the random graph process, $V(G_0)=[n]$. For the proof of Theorem\ref{thm:hamCores} we will need Lemma \ref{lem:expansion2} which replaces Lemma \ref{lem:expansion} in  the case that the underlying graph $G_i$ has sufficiently many edges. Its proof is found at Appendix \ref{app:lemmas2}.

\begin{lemma}\label{lem:expansion2}
W.h.p. for $k^{100}n\leq i\leq n\log n$,
\begin{itemize}
    \item[(i)] every set $S\subset V(G)$ of size  $|S| \leq \frac{3n}{\log^2 n}$  spans less than $1.1|S|+1$ edges in $G_i$,
\item[(ii)] there does not exists a set $S\subset V(S)$ of size  $\frac{n}{\log^2 n} \leq |S| \leq \frac{n}{100k}$ such that $N(S)\cup S$ induces a connected subgraph of $G_i$ and  $|N(S)| < k|S|$,
\item[(iii)] $|V(G_i^{(k)})| \geq (1-e^{\frac{i}{40n}})n$,
\item [(iv)] for every subgraph $F$ of $G_i$ of maximum degree $k-2$ the graph $G_i^{(k)}\setminus F$ spans a matching of size at least $0.5n-\frac{n}{\log\log n}$.
\end{itemize}
\end{lemma}

{\em Proof of Theorem\ref{thm:hamCores}(sketch):}  We consider 3 distinct intervals that partition $\{0,1,...,n(n-1)/2\}$.

\noindent \textbf{Case 1:} $0\leq i \leq k^{101}n$. The fact that $G_{i}^{(k)}$ is either empty or has order linear in $n$  and it is distributed as $G_{n,m}^{\delta \geq k}$ together with  Theorem \ref{thm:ham} implies that w.h.p. $G_i^{(k)}\in \cA_k'$ for $i\leq k^{101}n$.

\noindent \textbf{Case 2:} $k^{101} n\leq i \leq n\log n$. For this regime we condition on the events described at Lemma \ref{lem:expansion2} occurring.

We first reveal the edges of $G_{i/k}$ and then the edges of $G_i$ that are not incident to vertices of $V(G_{i/k}^{(k)})$. We let $F_i$ be the graph consisting of the edges revealed so far and $R$ be the set of edge of $G_i$ that have not been revealed yet. Observe that as every vertex outside $G_{i/k}^{(k)}$ is incident to the same set of edges in both 
$F_i$ and $G_i$ and $G_{i/k}^{(k)} \subseteq G_i^{(k)}$ we have that $V(F_i^{(k)})=V(G_i^{(k)})$ and $F_i^{(k)} \subseteq G_i^{(k)}$. In addition with $V_1=V(G_{i/k}^{(k)})$ part  (iv) of Lemma \ref{lem:expansion2} implies that $|V_1| \geq (1-e^{i/40n})n$. Moreover one can show that $R\geq 0.5i$  with probability $1-o(n^{-2})$ and $R$ is distributed uniformly among all set of edges spanned by $V_1$, of size $|R|$, that do not intersect $E(G_{i/k}^{(k)})$. We randomly split $R$ into $k-1$ sets $E_1,E_2,...,E_{k-1}$ each of size at least $\lfloor 0.5i/(k-1) \rfloor$.  

We then construct the Hamilton cycles of $G_{i}^{(k)}$ one by one. Having constructed Hamilton cycles $H_1,H_2,...,H_{j-1}$, $j< (k-2)/2$ of $G_i^{(k)}$ such that $H_\ell$ is spanned by $F_i^{(k)} \cup R_\ell$ for $\ell <j$ we construct a Hamilton cycle $H_j$ spanned by $F_j'\cup R_j$ where $F_j'=F_i^{(k)}\setminus \big( \cup_{\ell \in[j-1]} H_j\big)$. Observe that $F_j'$ has minimum degree 4. Let $R_j=\{e_1,e_2,...,e_{\tau_j}\}$, where $\tau_j=|R_j|\geq-1 + 0.5i/k$. For $\ell\in [0.5i-1]$ let $F_{j,\ell}=F_j'\cup\{e_1,e_2,...,e_\ell\}$, $P_\ell$ be a longest path of $F_{j,\ell}$ and $v$ one of the endpoints of $P_\ell$. Let $End(P_\ell,v)$ be the set of endpoints such that if $u\in End(P_\ell,v)$ then there exists a sequence of Posa rotations starting from $P_\ell$ that outputs a longest path in $F_{j,\ell}$ from $v$ to $u$. As shown in the proof of Claim \ref{claim:Ham:2} we have that  there exists  $S\subset End(P_\ell,v)\cup N(End(P_\ell,v))$ such that $S$ spans at least $1.1|S|$ edges of $F_i'$. In addition, P\'osa Lemma states (see \cite{abook}) states that $$|N(End(P_\ell,v)|<2|End(P_\ell,v)|.$$ Thus  Lemma \ref{lem:expansion2} implies that $|End(P,v)| \geq \frac{n}{100k}$, and therefore $|End(P_\ell,v)\cap V_1| \geq \frac{n}{1000k}$.

The rest of the argument is identical to the one used in the  proof of Lemma \ref{lem:builtingmatch} and gives that the probability that $F_i'\cup R_i$ does not span a Hamilton cycle is bounded by,
$$\Pr\bigg(Bin\bigg(0.5i,\frac{n}{2000k}\bigg)\leq n\bigg)\leq \Pr\bigg(Bin\bigg(0.5k^{101}n,\frac{n}{1000k}\bigg)\leq n\bigg)=o(n^{-2}).$$

In a similar manner in the case that $k$ is odd we can built the last 2-matching after pilling off the  $(k-3)/2$ Hamilton cycles. On the other hand in the case that $k$ is even, with $H=\cup_{1\leq j\leq(k-2)/2}H_j$, Lemma \ref{lem:expansion2} implies that  $F_i^{(k)}\setminus H$  spans a matching of size at least $0.5|V(G^{(k)}_i)|-\frac{n}{\log\log n}$.

\noindent \textbf{Case 3:} $n\log n < i \leq \binom{n}{2}$.  Case 2 implies that w.h.p. $G_{n\log n}^{(k)} \in \cA_k'$. Thus, since $G_i\subset G_{i+1}$ for $i\geq 0$ we have,
\begin{align*}
\Pr( \exists i\geq n\log n: G_i^{(k)}\notin \cA_{k}') \leq  \Pr( G_{n\log n} \notin \cA_{k}')+\Pr( G_{n\log n} \neq G_{n\log n}^{(k)})=o(1).    
\end{align*}
\qed

\begin{appendices}
\section{Proof of Lemma \ref{lem:expansion}}\label{app:lemmas}
Lemma \ref{lem:expansion} states the following.
\begin{lemma}
Let $G\sim G_{n,m}^{\delta \geq k}$, where $m=cn$ with $k/2<c=O(1)$ and let $\lambda$ be given by \eqref{eq:lambda}. Let $\beta_1,\gamma_1\in (0,0.1)$ be such that $\bfrac{9e^{1+\lambda} \lambda^2}{cf_k(\lambda)} \bfrac{\gamma_1\lambda }{c}^{0.1}<\frac{1}{2}$ and $[2(k+\lambda) + \log_2(\beta_1\gamma_1)+ 3]\beta_1  < 2(1-\beta_1)$. Then with probability $1-o(n^{-1})$,
\begin{itemize}
\item[(i)] every set $S\subset V(G)$ of size $ |S| \leq \gamma_1n$  spans less than $1.1|S|+1$ edges,
\item[(ii)] every set $S\subset V(S)$ of size  $|S| \leq \beta_1 \gamma_1 n$  is incident to less than $2(1-\beta_1)\gamma_1 n$ edges,
\item[(iii)] $G$ does not span a set of $\frac{2n}{(\log\log n)^6}$ vertex disjoint cycles.
\end{itemize}
\end{lemma}

\begin{proof}
For $\ell\geq 0$ let $\Phi(2\ell)$ be the number of ways to partition  a $2\ell$ element set into pairs. Then, for $1\leq \ell \leq m$,
\begin{align}\label{eq:boundmatchings}
\frac{\Phi(2m-2\ell)\Phi(2\ell)}{\Phi(2m)}&= 
\frac{\frac{(2m-2\ell)!}{(m-\ell)!2^{m-\ell}}\frac{(2\ell)!}{\ell!2^{\ell}}}{\frac{(2m)!}{m!2^{m}}}
  =  \frac{(2\ell)!}{\ell!} \frac{m!(2m-2\ell)!}{(m-\ell)!(2m)!}\nonumber 
  \\& \leq (2\ell)^\ell \bfrac{1}{2(2m-2\ell)}^\ell 
  =\bfrac{\ell}{2m-2\ell}^\ell.
\end{align}

For $s\in [n]$ let $r(s)=s+2$ if $s \leq (\log\log n)^8$ and $r(s)=1.1s$ if $ (\log\log n)^8 < s \leq \gamma_1n$. 
Let $\cE$ be the event that there does not exists $ S\subset V(G)$ of size $|S| \leq \gamma_1 n$  that spans at least $r(|S|)$ edges.

{\emph{(i)}}
For $s\leq \gamma_1 n$, $r(s)$ is smaller or equal to the smaller integer that is larger or equal to $1.1s +1$. Thus \eqref{eq:models} implies that the probability of (i) not occurring is bounded by,
\begin{align}
\Pr&\bigg( \exists S\subset V(G):  |S| \leq \gamma_1 n  \text{ and } S \text{ spans  at least $r(|S|)$ edges} \bigg)\nonumber
    \\&\leq O(n^{0.5}) \sum_{s= 4}^{\gamma_1n}  \binom{n}{s}  \sum_{\substack{d_1,d_2,...,d_s\geq k\\ z_1\leq d_1,...,z_s\leq d_s     \\ z_1+...+z_s=2r(s)}} 
    \prod_{i=1}^s \frac{\lambda^{d_i}}{ d_i! f_{k}(\lambda)} \binom{d_i}{z_i} 
    \frac{\Phi(2m-2r(s))\Phi(2r(s))}{\Phi(2m)} \label{exp1}
    \\&\leq O(n^{0.5})  \sum_{s= 4}^{\gamma_1 n} \binom{n}{s}   \frac{\lambda^{2r(s)}}{f_{k}^s(\lambda)} \sum_{\substack{d_1,d_2,...,d_s\geq k\\ z_1\leq d_1,...,z_s\leq d_s\\ z_1+...+z_s=2r(s)}} 
    \prod_{i=1}^s \frac{\lambda^{d_i-z_i}}{(d_i-z_i)!} 
   \bfrac{r(s)}{2m-2r(s)}^{r(s)} \nonumber
    \\&\leq  O(n^{0.5})  \sum_{s=4}^{\gamma_1 n} \binom{n}{s}  \frac{\lambda^{2r(s)}}{f_{k}^s(\lambda)}  
  \bfrac{r(s)}{2m-2r(s)}^{r(s)}
    \sum_{D\geq 2r(s)} \sum_{\substack{z_1,z_2,...,z_s\geq 0\\z_1+...+z_s=2r(s)}}
      \frac{\lambda^{D-2r}s^{D-2r(s)}}{(D-2r(s))!} \label{exp2}
     \\&\leq O(n^{0.5})  \sum_{s= 4}^{\gamma_1n} \binom{n}{s}  \frac{\lambda^{2r(s)}}{f_{k}^s(\lambda)} \bfrac{r(s)}{2m-2r(s)}^{r(s)} \sum_{D\geq 2r(s)} \binom{2r(s)+s-1}{s-1}  \frac{(\lambda s)^{D-2r(s)}}{(D-2r(s))!} \nonumber
    \\&\leq   O(n^{0.5})   \sum_{s= 4}^{\gamma_1 n}  \bfrac{en}{s}^s \frac{\lambda^{2r(s)}}{f_{k}^s(\lambda)} \bfrac{r(s)}{2m-2r(s)}^{r(s)}
    \bfrac{e(2r(s)+s)}{s}^se^{s \lambda } \nonumber
    \\&\leq   O(n^{0.5})   \sum_{s= 4}^{\gamma_1n }  \bfrac{en}{s}^s \frac{\lambda^{2r(s)}}{f_{k}^s(\lambda)} \bfrac{r(s)}{2m-2r(s)}^{r(s)}
    9^se^{s \lambda } \nonumber
    \\&\leq  O(n^{0.5}) \sum_{s=  \frac{n}{(\log\log n)^2}}^{\gamma_1n} 
\bfrac{9e^{1+\lambda} \lambda^2}{cf_k(\lambda)}^s \bfrac{\lambda s}{cn}^{0.1s}+ o(n^{-1}) =o(n^{-1}). \nonumber
\end{align}

{\textbf{Explanation of \eqref{exp1}}} We first choose $s$ vertices $v_1,v_2,...,v_s$ in $\binom{n}{s}$ ways. Those vertices will span a subgraph $S$ with $r$ edges. 
The degree of $v_i$ in $G$ will be $d_i$, this occurs with probability $\prod_{i=1}^s\frac{\lambda^{d_i}}{d_i! f_k(\lambda)}$, and its degree in $S$ will be $z_i$. Then, for each vertex $v_i$ we choose a set of $z_i$ out of the $d_i$ copies of $v_i$. The last term is the probability that those copies induce $\sum_{i=1}^sz_i/2$ edges when we pass form the sequence in $[n]^{\sum_{i\in [n]}d(i) }$ to the corresponding graph.

To derive \eqref{exp2} we used the following identity. For fixed $z_1,z_2,...,z_s$ if $\sum_{i=1}^s z_i=2r$ and $\sum_{i=1}^s d_i-z_i=D-2r$ then  $\sum_{\substack{z_1\leq d_1,...,z_s\leq d_s\\ d_1+...+d_s=D \\ z_1+z_2+...+z_s=2r}}\frac{(D-2r)!}{\prod_{i=1}^s(d_i-z_i)!}= s^{D-2r}$. 

{\emph{(ii)}} Let $Y\sim Po_{\geq k}(\lambda)$. Then, for $i\geq 0$
\begin{align*}
    \Pr(Y\geq 2(\lambda+k)+i)
    \leq \frac{\frac{\lambda^{2(\lambda+k)+i}}{[2(\lambda+k)+i]!}}{\frac{\lambda^{2(\lambda+k)}}{[2(\lambda+k)]!}} \leq  \frac{\lambda^{i}}{[2(\lambda+k)]^i)}\leq 2^{-i}.
\end{align*}
Thus, \eqref{eq:models} implies,
\begin{align*}
    \Pr&(\exists i\geq 1: \text{ there exists more than $2^{3-i}n$ vertices in $G$ of degree } 2(\lambda+k)+i)
    \\&\leq O(n^{0.5})\sum_{i\geq 1 }\Pr(Binomial(n,2^{-i})\geq 2^{3-i}n) \leq O(n^{0.5})\sum_{i\geq 1 } \binom{n}{2^{3-i}n} (2^{-i})^{2^{3-i}n}
    \\&\leq O(n^{0.5})\sum_{i\geq 1 } \bfrac{ en \cdot 2^{-i}}{2^{3-i}n}^{2^{3-i}n}=o(n^{-1}).
\end{align*}
Hence with probability $1-o(n^{-1})$ any set of at most $\beta_1 \gamma_1 n$ vertices spans at most 
\begin{align*}
    &\sum_{i\geq  0} [2(k+\lambda) - \log_2(\beta_1\gamma_1)+i+1]\beta_1\gamma_1 n \cdot 2^{-i} 
    \\&\leq [2(k+\lambda) - \log_2(\beta_1\gamma_1)+3]\beta_1\gamma_1 n
    \leq 2(1-\beta_1)\gamma_1 n
\end{align*}
edges.

{\emph{(iii)}} 
For a set $S$ and an integer $k\in \mathbb{Z}_{\geq 0}$ denote by $N_{\leq k}(S)$ the set of vertices that are at distance at most $k$ from some vertex in $S$. Call a cycle \emph{small} if it has size at most $(\log\log n)^6$ and let $X_{small}$ be the number of small cycles in $G$. In the event $\cE$ we have that there does not exists a small cycle $C$ such that $N_{\leq \log_{1.1}\log n}(V(C))$ spans more than 3 small cycles (including $C$) as such a cycle would give rise to a set of size $s\leq 3(\log\log n)^6+2\log_{1.1}\log n$ that spans $s+2$ edges. Therefore in the event $\cE$ there exists a set $\cC$ of at least $X_{small}/3$  cycles such that the sets $N_{\leq \log_{1.1}\log n}(V(C)), C\in \cC$ are disjoint. Furthermore, in the event $\cE$, as $G$ has minimum degree at least 3 we have that 
 $$|N_{\leq \log_{1.1}\log n}(V(C))| \geq 1.1^{\log_{1.1}\log n} \geq \log n \text{ for } C\in \cC$$
and therefore $X_{small}\leq \frac{3n}{\log n}$. Hence with probability at least $\Pr(\cE)=1-o(n^{-1})$  any set of vertex disjoint cycles in $G$ consists of at most $\frac{3n}{\log n}+\frac{n}{(\log n\log n)^6} \leq \frac{2n}{(\log \log n)^6}$ many cycles.
\end{proof}

\section{Proof of Lemma \ref{lem:expansion2}}\label{app:lemmas2}
Recall, Lemma \ref{lem:expansion2} states,

\begin{lemma}
W.h.p.  for $k^{100}n\leq i\leq n\log n$,
\begin{itemize}
    \item[(i)] every set $S\subset V(G)$ of size  $|S| \leq \frac{3n}{\log^2 n}$  spans less than $1.1|S|+1$ edges in $G_i$,
\item[(ii)] there does not exists a set $S\subset V(S)$ of size  $\frac{n}{\log^2 n} \leq |S| \leq \frac{n}{100k}$ such that $N(S)\cup S$ induces a connected subgraph of $G_i$ and  $|N(S)| < k|S|$,
\item[(iii)] $|V(G_i^{(k)})| \geq (1-e^{\frac{i}{40n}})n$,
\item [(iv)] for every subgraph $F$ of $G_i$ of maximum degree $k-2$ the graph $G_i^{(k)}\setminus F$ spans a matching of size at least $0.5|V(G_i^{(k)})|-\frac{n}{\log\log n}$.
\end{itemize}
\end{lemma}
\begin{proof} {\emph{(i)}}
Let $\mathcal{S}_i$ be the event that there exists $ S\subset[n]$ of size $|S| \leq \frac{3n}{\log^2 n}$ that  spans at least  $\lceil 1.1|S|+1 \rceil \geq |S|+2$ edges in $G_i$. Then,
\begin{align*}
&\Pr\big(\exists i \in [k^{100}n, n\log n]: \mathcal{S}_i  \text{ occurs}\big)=\Pr( \mathcal{S}_{n\log n}  \text{ occurs})
\\&\leq\sum_{s=4}^{\frac{3n}{\log^2 n}} \binom{n}{s} \binom{s^2}{\lceil1.1s+1 \rceil} \bfrac{3\log n}{n}^{\lceil1.1s+1 \rceil}
\leq\sum_{s=4}^{\frac{3n}{\log^2 n}} \bfrac{en}{s}^s \bfrac{es\log n}{1.1n}^{\lceil1.1s+1 \rceil}=o(1).
\end{align*}
\emph{(ii)} For $i\geq 0$ let $p_i=2i/n(n-1)$.
\begin{align*}
    \Pr(\neg(ii))&\leq O(n^{0.5})\sum_{i=k^{100}n}^{n\log n} \sum_{s=\frac{n}{\log^2 n} }^{\frac{n}{100k}} \sum_{t=1}^{ks}   \binom{n}{s+t}\binom{s+t}{s} (s+t)^{s+t-2}p_i^{s+t-1}(1-p_i)^{s(n-s-t)} 
    \\&\leq  \sum_{i=k^{100}n}^{n\log n} \sum_{s=\frac{n}{\log^2 n} }^{\frac{n}{100k}} \sum_{t=1}^{ks}   (2en)^{s+t}p^{s+t-1}_ie^{-0.5snp_i}
    \leq  \sum_{i=k^{100}n}^{n\log n} \sum_{s=\frac{n}{\log^2 n} }^{\frac{n}{100k}} ks   (2enp_i)^{(k+1)s}p_i^{-1}e^{-0.5snp_i}
\\&  = \sum_{i=k^{100}n}^{n\log n} \sum_{s=\frac{n}{\log^2 n} }^{\frac{n}{100k}} \frac{ks}{p_i} \bigg( (2enp_i)^{(k+1)}e^{-np_i}\bigg)^s=o(1).
\end{align*}
\emph{(iii)}
In the event that $|V(F_{i/k}^{(k)})| \leq (1-e^{\frac{i}{40n}})n$ then there exists $S\subset [n]$ of size $e^{-\frac{i}{40kn}} n$ such that every vertex in $S$ has at most $k-1$ neighbors in $[n] \setminus S$. 
Therefore,
\begin{align*}
    \Pr(\neg (iii)) &\leq \sum_{i=k^{100}n}^{n\log n} \binom{n}{e^{-i/40kn}n}\bigg[\sum_{j=0}^{k-1}\binom{n}{j}p_i^j(1-p_i)^{(1-e^{-i/40kn}n-j)}\bigg]^{e^{-\frac{i}{40kn}n} }
    \\&\leq \sum_{i=k^{100}n}^{n\log n}\bigg(e^{(1+i/40kn)}(np_i)^ke^{-0.9np_i}\bigg)^{e^{-\frac{i}{40kn}n}}=o(1).
\end{align*}
\emph{(iv)} For a subgraph $F$ of $G_{n\log n}$ of maximum degree $k-2$ let $\mathcal{M}(F,i)$ be the event
that the graph $G_i^{(k)}\setminus F$ does not span a matching of size $0.5|V(G_i^{(k)})|-\frac{n}{\log\log n}$. As in the proof of Lemma \ref{lem:expansion}, part (i) of Lemma \ref{lem:expansion2} implies that w.h.p. there does not exists $i\leq n\log n$ such that $G_i$ spans a set of $\frac{n}{(\log\log n)^2}$ edge disjoint cycles. Thereafter as in the proof of Lemma \ref{lem:pillmatch} we have that w.h.p. the event $\mathcal{M}(F,i)$ implies that there exist $s\geq \frac{n}{2\log\log n}$ and disjoint sets $S,T\subset [n]$ of size $s$ such that in $G_i$ every vertex in $S$ has at least 2 neighbors in $T$ and at most $k-2$ neighbors in $[n]\setminus T$ (these neighbors are defined by $F$). Therefore,
\begin{align*}
    &Pr( \mathcal{M}(F,i) \text{ occurs  for some pair }F,i)
    \\&\leq \sum_{i=k^{100}n}^{n\log n}\sum_{s=\frac{n}{2\log\log n}}^{n/2} \binom{n}{2s}2^{2s} \bigg[\binom{s}{2}p_i^2\bigg]^s
    \bigg[\sum_{j=0}^{k-2}\binom{n}{j}p_i^j\bigg]^s(1-p_i)^{\binom{s}{2}+(n-2s)s-(k-2)s}
    \\&\leq \sum_{i=k^{100}n}^{n\log n}\sum_{s=\frac{n}{2\log\log n}}^{n/2} \bfrac{enp_i}{2}^{2s}
(np_i)^{(k-2)s}e^{-0.1snp_i}
\\&= \sum_{i=k^{100}n}^{n\log n}\sum_{s=\frac{n}{2\log\log n}}^{n/2} \bigg[ \bfrac{e^2(np_i)^k e^{-0.1np_i}}{4} \bigg]^s=o(1).
\end{align*}

\end{proof}

\end{appendices}


\begin{thebibliography}{9}

\bibitem{AKS}
 M. Ajtai, J. Koml\'os and E. Szemer\'edi, The first occurrence of Hamilton cycles in
random graphs, \emph{Annals of Discrete Mathematics} 27 (1985) 173-178.

\bibitem{A} M. Anastos, On a $k$-matching algorithm and finding $k$-factors in random graphs with minimum degree $k+1$ in linear time.
 

\bibitem{AF}
 M. Anastos and A. M. Frieze.  Hamilton cycles in random graphs with minimum  degree  at  least  3: an  improved  analysis, \emph{ Random Structures and Algoriths} 57 (2020), 865-878.  
  
\bibitem{Bol}
B. Bollob\'as, The evolution of sparse graphs,  \emph{Graph theory and combinatorics}  (Cambridge, 1983), Academic Press, London, (1984) 3557.



\bibitem{BolCM}
B. Bollob\'{a}s, A probabilistic proof of an asymptotic formula for the number of labelled regular graphs, \emph{European Journal on Combinatorics} 1 (1980), 311–316.

\bibitem{BCFF}
B. Bollob\'as, C. Cooper, T. Fenner, and A. Frieze, On Hamilton cycles in sparse random graphs with minimum degree at least k. \emph{Journal of Graph Theory} 34 (2000), 42-59.

\bibitem{BKV}
 B. Bollob\'as,  J.H. Kim and J. Verstra\"{e}te, Regular subgraphs of random graphs, \emph{Random Structures \& Algorithms} 29 (2006), 1-13.


\bibitem{F}
A.M. Frieze, On a greedy 2‐matching algorithm and Hamilton cycles in random graphs with minimum degree at least three, \emph{Random Structures and Algorithms} 45 (2014): 443-497.

\bibitem{abook}
A. Frieze and M. Karo\'nski, Introduction to random graphs, Cambridge University Press, (2015).



 
\bibitem{Komlos}
J. Koml\'os and E. Szemer\'edi, Limit distribution for the existence of Hamilton cycles in random graphs, \emph{Discrete Math}, 43 (1983), 55-63.

\bibitem{KLS}
M. Krivelevich, E. Lubetzky, and B. Sudakov, Cores of random graphs are born Hamiltonian, \emph{Proceedings of the London Mathematical Society} 109 (2014) 161-188.

\bibitem{Kor}
A. Korshunov, Solution of a problem of Erd\H{o}s and R\'enyi on Hamilton cycles non-oriented
graphs, \emph{ Soviet Math. Dokl.}, 17 (1976), 760–764.

\bibitem{l} 
T. {\L}uczak, Size and connectivity of the k-core of a random graph. \emph{Discrete Mathematics} 91.1 (1991): 61-68.

\bibitem{Mc}
C. McDiarmid, On the method of bounded differences, \emph{In Surveys in Combinatorics},  Cambridge University Press, Cambridge, 1989, 148–188.

\bibitem{McK} B. McKay, Asymptotics for symmetric 0-1 matrices with prescribed row sums, \emph{Ars Combinatoria}, 19A (1985), 15–25.

\bibitem{PSW}
B. Pittel, J. Spencer and N. Wormald, Sudden emergence of a giant $k$-core in a random graph, \emph{J. Combinatorial Theory, Series B} 67 (1996), 111-151.

\bibitem{Posa}
L. P\'osa, Hamiltonian circuits in random graphs, \emph{Discrete Mathematics} 14 (1976) 359–364.


\end{thebibliography}
\end{document}